\documentclass[12pt, oneside]{amsart}
\usepackage{amssymb}  % for \leqslant, \geqslant
\usepackage{amsmath} % for \pmod with smaller space before (mod ...)
\usepackage{amsthm}  % for \begin{proof} \end{proof}   !!! defines \openbox !!!
\usepackage{graphicx}
\usepackage{color}

\usepackage[utf8]{inputenc}
\usepackage[english]{babel}
\usepackage{cite, verbatim}
\usepackage{hyperref}

\newtheorem{theorem}{Theorem}
\newtheorem{problem}[theorem]{Problem}
\newtheorem{lemma}{Lemma}
\newtheorem{proposition}{Proposition}
\numberwithin{lemma}{section}
\numberwithin{proposition}{section}

\newcommand{\Aut}{\operatorname{Aut}}
\newcommand{\Out}{\operatorname{Out}}
\newcommand{\Soc}{\operatorname{Soc}}

\makeatletter
\@namedef{subjclassname@2020}{\textup{2020} Mathematics Subject Classification}
\makeatother

\begin{document}

\title[On groups isospectral to groups with abelian Sylow 2-subgroups]{On finite groups isospectral to groups \\with abelian Sylow 2-subgroups}

\author{M.\,A.\,Grechkoseeva and  A.\,V.\,Vasil'ev}

\address{Sobolev Institute of Mathematics, Novosibirsk, Russia}
\email{grechkoseeva@gmail.com}
\email{vasand@math.nsc.ru}
%\thanks{Grants}


\begin{abstract} The spectrum of a finite group is the set of orders of its elements. We are concerned with finite groups having the same spectrum as a direct product of nonabelian simple groups with abelian Sylow 2-subgroups. 
For every positive integer $k$, we find $k$ nonabelian simple groups with abelian Sylow 2-subgroups such that their direct product is uniquely determined by its spectrum in the class of all finite groups. On the other hand, we prove that there are infinitely many finite groups having the same spectrum as the direct cube of the small Ree group $^2G_2(q)$, $q>3$, or the direct fourth power of the sporadic group~$J_1$. 

%\noindent\textsc{Key words:}  simple group, recognition by spectrum, 
%small Ree group, sporadic Janko group

%\noindent\textsc{MSC:} 20D60, 20D06.
 \end{abstract}

\maketitle

\section*{Introduction}

Given a finite group $G$, the {\em spectrum} $\omega(G)$ is the set of the orders of elements of $G$. A~group $H$ is {\em isospectral} to $G$ if $\omega(H)=\omega(G)$.  We say that $G$ is {\em recognizable (by spectrum)} if every finite group isospectral to $G$ is isomorphic to $G$, {\em almost recognizable} if there are finitely many pairwise nonisomorphic finite groups isospectral to $G$, and {\em unrecognizable} if there are infinitely many such groups. Determining whether $G$ is recognizable, or almost recognizable, or unrecognizable, or, more generally,  describing groups isospectral to $G$, is known as the recognition of $G$ by spectrum. The problem of recognition by spectrum has become quite popular in recent decades, see the state of the art in the recent survey~\cite{23Survey}.

By the celebrated theorem of V.~D. Mazurov and W. Shi in \cite{12MazShi.t}, a finite group $G$ is unrecognizable if and only if $G$
is isospectral to a group having a nontrivial normal abelian subgroup. In particular, every group with nontrivial solvable radical is unrecognizable, and so the recognition problem is of interest only for groups having the following structure: 
\begin{equation}\label{eq:structure}
L_1\times L_2\times \dots\times L_k=\Soc(G)\leq G\leq \Aut (L_1\times L_2\times \dots\times L_k)
\end{equation}
for some nonabelian simple groups $L_i$, $1\leqslant i\leqslant k$. Here, as usual, the socle $\Soc(G)$ is the subgroup generated by all minimal normal subgroups of~$G$.

In the present paper, continuing investigations of~\cite{93BrShi,24LMVW}, we consider the problem of recognition by spectrum for finite groups with abelian Sylow $2$-subgroups. These groups were described by J.~H. Walter in \cite{69Wal}. In particular, if such a group $G$ satisfies \eqref{eq:structure}, then $|G:\Soc(G)|$ is odd and every $L_i$ is either a~$2$-dimensional linear group $L_2(q)=PSL_2(q)$, or a small Ree group $^2G_2(q)$, or the sporadic Janko group~$J_1$.
%(henceforth the {\em Atlas one-letter} notation for simple groups is used, see \cite{85Atlas}). 
Observe that there are further restrictions on $q$ for $L_2(q)$, see Lemma \ref{l:WalSimple} below.
Also recall that the group $^2G_2(q)$ is defined only for $q=3^\alpha$, $\alpha$ odd, and simple exactly when $q>3$.

It is known that a recognizable group can have arbitrarily many nonabelian composition factors \cite[Theorem 4.6]{23Survey}. More surprisingly, such a group can have arbitrarily many pairwise isomorphic nonabelian composition factors \cite{23YGSV}. So the first question we address here is how many nonabelian composition factors a recognizable group with abelian Sylow 2-subgroups can have, and the answer is `arbitrarily many'.

 \begin{theorem}\label{t:inf}
There is an infinite increasing sequence of odd primes $p_i, i\in\mathbb{N},$ such that for every $k\in\mathbb{N}$ and
$$P_k=\prod_{i=1}^k {}^2G_2(3^{p_i}),$$
the equality $\omega(G)=\omega(P_k)$ yields $G\simeq P_k$.
\end{theorem}

Despite its `positivity', Theorem \ref{t:inf} shows that it is quite hard to expect the solution of the recognition problem for all finite groups with abelian Sylow $2$-subgroups (cf. \cite[Problem~1.4]{24LMVW}). It seems that the reasonable goal here is to solve this problem for the direct powers of simple groups. 

Suppose that $L$ is a nonabelian simple group with abelian Sylow $2$-subgroups. The group $L$ itself is  always recognizable~\cite{93BrShi}. The group $L\times L$  is recognizable if and only if $L$ is either $^2G_2(q)$ or $J_1$ \cite{24LMVW}. In fact,  $L_2(q)\times L_2(q)$ is  unrecognizable for all $q$, regardless of whether Sylow $2$-subgroups of $L_2(q)$  are abelian or not \cite[Proposition~1.3]{24LMVW}. We prove that the direct cube of $^2G_2(q)$ and the fourth direct power of $J_1$ are unrecognizable. 

\begin{theorem}\label{t:cube} Let $L$ and $k$ satisfy one of the following:
\begin{enumerate}
 \item $L={}^2G_2(q)$ and $k=3$;
 \item $L=J_1$ and $k=4$.
\end{enumerate}
Then there are infinitely many finite pairwise nonisomorphic groups $G$ with $\omega(G)=\omega(L^k)$.
\end{theorem}

Since unrecognizability of $L^k$ implies unrecognizability of $L^m$ for all $m>k$, to solve the recognition problem  for all direct powers of simple groups with abelian Sylow 2-subgroups, it remains to answer the following question (note that $J_1\times J_1\times J_1$ is almost recognizable if and only if it is recognizable, see  Proposition \ref{p:J1} below). 

\begin{problem}\label{pr:J1}
Is the group $J_1\times J_1\times J_1$ recognizable by spectrum? 
\end{problem}

In conclusion, we observe that there are no results on recognition of 
 groups $G$ that have abelian Sylow 2-subgroups and satisfy \eqref{eq:structure} with $G>\Soc(G)$. In particular, nothing is known even in the case $k=1$, that is, when $G$ is  almost simple but not simple.

\section{Preliminaries}

As usual, $\pi(G)$ denotes the set of prime divisors of $|G|$. The set $\omega(G)$ is closed under taking divisors, and so uniquely determined by the subset $\mu(G)$ consisting of numbers maximal under divisibility. The {\em prime graph} $GK(G)$ is the graph whose vertex set is equal to $\pi(G)$ and two different vertices $p,q\in\pi(G)$ are adjacent if and only if $pq\in\omega(G)$. A {\em coclique} of $GK(G)$ is a set of pairwise nonadjacent vertices, and we denote by $t(G)$ the maximum size of cocliques in $GK(G)$.

Let $r$ be a prime. If $G$ is a finite group, we write $O_r(G)$ and $O_{r'}(G)$ for the largest normal $r$-subgroup and $r'$-subgroup of $G$, respectively. Also $O^{r'}(G)$ denotes the smallest normal subgroup of $G$ for which the quotient is an $r'$-group. If $a$ is a positive integer, then $(a)_r$ is the highest power of $r$ dividing $a$ and $(a)_{r'}=a/(a)_r$. Given positive integers $a_1,\dots, a_k$, we write $(a_1,\dots, a_k)$ and $[a_1,\dots,a_k]$ for their greatest common divisor and least common multiple, respectively.

\begin{lemma}[{Bang \cite{86Bang}, Zsigmondy \cite{Zs}}]\label{l:bzs}
For every positive integers $a,i>1$, one of the following holds:
\begin{enumerate}
 \item there is a prime $r$ that divides $a^i-1$ but does not divide
 $a^j-1$ for all $1\leq j<i$; 
 \item either $i=6$, $a=2$, or $i=2$, $a=2^t-1$ for some $t$.
\end{enumerate}
\end{lemma}

A prime $r$ satisfying (a) of Lemma \ref{l:bzs} is called a {\em primitive prime divisor} of $a^i-1$. 

The following lemma is well known. 

\begin{lemma} \label{l:num_th} Let $a, i,j$ be integers, $|a|>1$ and $i,j>0$. Then the following hold:
\begin{enumerate}
 \item $(a^i-1,a^j-1)=a^{(i,j)}-1$;
 \item $(a^i-1,a^j+1)$ is equal to $(2,a-1)$ if $(i)_2\leq (j)_2$ and $a^{(i,j)}+1$ otherwise;
 \item If an odd prime $r$ divides $a-1$, then $(a^i-1)_r=(i)_r(a-1)_r$.
\end{enumerate}
\end{lemma}

As we said in Introduction, the structure of finite groups with abelian Sylow $2$-subgroups  were established (up to some details covered only by the classification of finite simple groups) in~\cite{69Wal}. 

\begin{lemma}\label{l:WalSimple}  If $G$ is a nonabelian simple group with an abelian Sylow $2$-subgroup $S$, then $S$ is elementary abelian and $G$ is one of the following groups:
\begin{enumerate}
\item the linear groups $L_2(q)$,  $q>3$, $q\equiv 3, 5 \pmod{8}$ or $q=2^\alpha$;
\item the small Ree groups $^2G_2(q)$,  $q=3^\alpha$,  $\alpha\geq 3$ is odd;
\item the sporadic Janko group $J_1$.
\end{enumerate}
\end{lemma}

\begin{lemma}\label{l:Walter}{\em\cite[Theorem~I]{69Wal}} If $G$ is a finite group with abelian Sylow $2$-subgroups, then the group $O^{2'}(G/O_{2'}(G))$ is a direct product of an abelian $2$-group and some nonabelian simple groups listed in Lemma~{\em\ref{l:WalSimple}}.
\end{lemma}

We will need some information about simple groups with abelian 2-Sylow subgroups. 

\begin{lemma}%{\rm\cite[Lemma~3.7]{24LMVW}}
\label{l:ree} 
Let $G={}^2G(q)$, $q=3^\alpha>3$. Then the following hold:
\begin{enumerate}
 \item $|G|=q^3(q-1)(q^3+1)$;
 \item $|\Out G|=\alpha$; 
 \item $\mu(G)=\{6, 9, q-1, (q+1)/2, q-\sqrt{3q}+1, q+\sqrt{3q}+1\}$.
\end{enumerate}
\end{lemma}

\begin{proof}
See \cite{61ReeG} for (a) and (b) and \cite[Lemma 4]{93BrShi} for (c).
\end{proof}

\begin{lemma}{\rm\cite[Lemma~3.8]{24LMVW}}\label{l:ReeRep}
Let $G={}^2G_2(q)$, where $q>3$, and  let $g\in G$ be an $r$-element for a prime~$r\neq 3$. If $G$ acts faithfully on a $p$-group $V$ for some prime $p$, then the coset $Vg$ of the natural semidirect product $V\rtimes G$ contains an element of order~$p|g|$.
\end{lemma}

%The spectra of $L_2(q)$ and $J_1$ are well known and we just write them up for %convenience. 

\begin{lemma} \label{l:lin}  Let $G=L_2(q)$, where $q>3$ is power of a prime $p$. Then the following hold:
\begin{enumerate}
 \item $\mu(G)=\{p, (q-1)/d, (q+1)/d\}$, where $d=(q-1,2)$;
 \item if $G\leq T\leq \Aut G$, then $t(T)\leq 3$.
\end{enumerate}
\end{lemma}

\begin{proof} (a) See, for example, \cite[Haupsatz 8.27]{67Hup}.

(b) Suppose that $\sigma$ is a coclique of size 4 in $GK(T)$. Since $\Out G$ is abelian, it follows that $|\sigma\cap\pi(G)|\geq 3$. Then (a) implies that $|\sigma\cap\pi(G)|=3$ and $p\in\sigma$. If $r\in \pi(G)\setminus\sigma$, then $r$ is the order a field automorphism of $G$. But every field automorphism of $G$ centralizes an element of order $p$, and so $pr\in\omega(T)$, a contradiction.
\end{proof}

\begin{lemma}\label{l:J1}
 $\mu(J_1)=\{6,7, 10,11, 15,19\}$.
\end{lemma}

\begin{proof}
 See, for example, \cite{85Atlas}.
\end{proof}

We conclude with a series of lemmas concerning element orders in group extensions.

\begin{lemma}{\rm\cite[Lemma~1.1]{05Vas.t}}\label{l:three}
Suppose that a finite group $G$ have a normal series of subgroups
$1 \leq K \leq M \leq G$, and distinct primes $p$, $q$ and $r$ are such that $p$ divides $|K|$, $q$ divides $|M/K|$, and $r$ divides $|G/M|$. Then at least one of the numbers $pq$, $pr$, and $qr$ belongs to $\omega(G)$.
\end{lemma}

\begin{lemma}\label{l:hh}
Let $G$ be a semidirect product of a $t$-group $T$ and a group $\langle x\rangle$ of order $r$, where $t$ and $r$ are different odd primes, and let $[T,x]\neq 1$. If $G$ acts faithfully on a $p$-group $V$, where $p$ does not divide $|G|$, then $C_V(x)\neq 1$.  
\end{lemma}

\begin{proof}
We may assume that $V$ is elementary abelian and regard $V$ as $G$-module over a field of order $p$. Since  $[T,x]\neq 1$, it follows that $O_r(G)=1$. Now we apply \cite[Theorem IX.6.2]{82HupBl2}.
\end{proof}

\begin{comment}
\begin{lemma}
\label{l:ReeRep}
Let $G=R(q)$, where $q>3$, and  let $g\in G$ be an $r$-element for  some prime~$r\neq 3$. If $G$ acts on a $p$-group $V$ for some prime $p\neq r$,  then $C_V(g)\neq 1$.  
\end{lemma}

\begin{proof}
If $p=3$, then the claim follows from \cite{03GurTie}, so suppose that $p\neq 3$. We may assume that $V$ is elementary abelian and regard $V$ as $G$-module over a field of order $p$. Then $g$ acts fixed-point-freely on $V$ if and only it acts so on every irreducible component of $V$, so we may assume that $V$ is irreducible and faithful. By \cite[Theorem 1.1]{22TiepZal}, the degree of the minimal polynomial $f(x)$ of $g$ in the representation on $V$ is equal to $|g|$. Hence  $f(x)$ satisfies (b) of Lemma \ref{l:spectrumVG}, and so $g$ has a nontrivial fixed point in $V$.
\end{proof}
\end{comment}

\begin{lemma}\label{l:three_primes}
Let $G$ be a quasisimple finite group acting on a $p$-group $V$  for some prime $p$ not dividing $|G|$. Suppose that $\sigma\subseteq\pi(G)\setminus \{2\}$ and for every $r\in \sigma$,  a Sylow $r$-subgroup of $G$ acts fixed-point-freely on $V$. Then $|\sigma|\leq 2$.
\end{lemma}

\begin{proof}
We may assume that $V$ is a faithful irreducible $FG$-module for some field $F$ of characteristic $p$. Furthermore, we may assume that $F$ is a splitting field for $G$. By \cite[Theorem 3.3.3]{68Gor}, a Sylow $r$-subgroup of $G$ is cyclic. Now the conclusion follows from \cite[Theorem~1.1]{99Zal}.
\end{proof}

\begin{lemma}\label{l:spectrumVG}
Suppose that $G$ is a finite group, $V$ is a vector space over a field of positive characteristic $p$ and and $\varphi: G\rightarrow GL(V)$ is a representation of $G$. Then  the coset $Vg$ of the semidirect product $V\rtimes_\varphi G$ contains an element of order $p|g|$ if 
 and only if the Jordan form of $\varphi(g)$ has a unipotent block of size $(|g|)_p$.
\end{lemma}

\begin{proof} Let $v\in V$ and $k=|g|$. Then $(gv)^k=v+v^{\varphi(g)}+\dots+v^{\varphi(g)^{k-1}}$, and hence $Vg$ contains an element of order $pk$ if and only if $x^{k-1}+\dots+1$ does not annihilate $\varphi(g)$. This is equivalent to saying that the minimal polynomial $f(x)$ of $\varphi(g)$ does not divide $(x^k-1)/(x-1)$. Since $f(x)$ divides $x^k-1$ and $1$ is a root of $x^k-1$ of multiplicity $(k)_p$, the result follows.
\end{proof}

\section{Recognizable direct products}

In this section we will prove Theorem \ref{t:inf}. Define primes $p_i$ as follows.
\begin{equation}\label{eq:example}
 \text{\begin{minipage}{0.92\textwidth} Let $p_1=5$. For $i>1$, define $p_{i}$ to be the smallest prime larger than $p_{i-1}$ and not dividing $|{}^2G_2(3^{p_j})|$ for any $1\leq j<i$.\end{minipage}}
\end{equation}

We claim that this sequence $\{p_i\}_{i\in\mathbb{N}}$ is as required. Set $q_i=3^{p_i}$ and $R_i={}^2G_2(q_i)$. Given  $k$ a positive integer $k$, let  $P_k=\prod_{i=1}^kR_i$ and suppose that $G$ is a finite group such that $\omega(G)=\omega(P_k)$. We need to prove that $G\simeq P_k$.

Let $K=O_{2'}(G)$ and let $H$ be the preimage of $O^{2'}(G/K)$ in $G$. 
Set $\overline G=G/K$ and $\overline H=H/K$.
By Lemma \ref{l:Walter}, it follows that $\overline H=S\times A$, where $S$ is a direct product of nonabelian simple groups listed in Lemma \ref{l:WalSimple} and $A$ is an abelian 2-group.  

\begin{lemma}\label{l:primes} Let $i\geq 1$.
\begin{enumerate}
 \item $\pi(R_i)\cap\pi(R_j)=\{2,3,7\}$ for all $j\neq i$.
 \item $p_i\not\in\pi(R_j)$ for all $j$, and in particular, $5\not\in \pi(R_j)$ for all $j$. 
 \item $\pi(q_i\pm 1)\neq \{2\}$,  $7$ divides one of $q_i\pm \sqrt{3q_i}+1$, and $\pi(q_i\pm\sqrt{3q_i}+1)\neq \{7\}$.
 \item If $\pi({}^2G_2(3^m))\subseteq\pi(P_k)$ for $m>3$, then $m$ is one of $\{p_1,\dots,p_k\}$.
 \end{enumerate}
 
\end{lemma}

\begin{proof}
(a) Let $3\neq r\in \pi(R_i)\cap\pi(R_j)$ for $j\neq i$. It follows from Lemma \ref{l:ree}(a) that $r$ divides $((q_i-1)(q_i^{3}+1),(q_j-1)(q_j^{3}+1))$. Applying Lemma \ref{l:num_th}, we see that $(q_i-1, q_j-1)=(q_i-1, q_j^{3}+1)=2$ and $(q_i^3+1,q_j^3+1)=3^3+1$. 

(b) If $i>j$, then the claim is a part of the definition in \eqref{eq:example}. Suppose that $i\leq j$. Then $p_i\leq p_j$, and so $(p_i-1, p_j)=1$.
In particular, $3^{p_i-1}-1$ and $3^{p_j}-1$ are coprime, while  $(3^{p_i-1}-1,3^{3p_j}+1)$ is a divisor of $3^3+1$. Thus if $p_i\in\pi(R_j)$, then applying Fermat Little Theorem and the preceding observations, we see that $p_i=7$.  But $p_i\neq 7$ by  construction.

(c) Since $7$ divides $q_i^3+1$ and is coprime to $q_i+1$, it divides one of $q_i\pm \sqrt{3q_i}+1$.

Note that $(q_i-1)_2=2$, $(q_i+1)_2=4$ and $q_i\geq 3^5$, so $(q_i\pm 1)>(q_i\pm 1)_2$. Similarly, using Lemma \ref{l:num_th} and the fact that $p_i\neq 3, 7$, we calculate that $(q_i^3+1)_7=(3^{3p_i}+1)_7=(3^3+1)_7=7,$ and since  $q_i\pm \sqrt{3q_i}+1>7$, the last  claim follows. 

(d) Let $r$ be a primitive prime divisor of $3^{3m}+1$. Arguing as in (a), we see that $r$ does not divide $q_i-1$ for all $i$. If $r$ divides $q_i^3+1$, then $m$ divides $p_i$ and so $m=p_i$.
\end{proof}

%Using Lemma \ref{l:num_th}(i), it is not difficult to show that 
%\begin{equation}\label{e:dif}\pi(R_i)\cap\pi(R_j)=\{2,3,7\} \text{ for all }i\neq j,%\end{equation}
%\begin{equation} \label{e:aut} p_i\not\in\pi(R_j)\text{ for all $i$ and $j$}.%\end{equation}                                                              

Let $i\in\{1,\ldots,k\}$. By Lemma \ref{l:primes}(c), we can take odd $r_{i1}\in\pi(q_i-1)$,   $r_{i2}\in\pi(q_i+1)$, $r_{i3}\in\pi(q_i-\sqrt{3q_i}+1)$, and $r_{i4}\in\pi(q_i+\sqrt{3q_i}+1)$ with $r_{i3},r_{i4}\neq 7$. By Lemma~\ref{l:ree}(c), the set $\sigma_i=\{r_{i1},r_{i2},r_{i3},r_{i4}\}$ is a coclique of size $4$ in $GK(R_i)$. Applying  Lemma \ref{l:primes}(a), we see that $\sigma_i$ remains to be a coclique in $GK(P_k)$.

\begin{lemma}\label{l:unique}
There is a unique direct factor $S_i$ of $S$ such that $|\sigma_i\cap\pi(S_i)|\geq 3$. The group $\overline G$ normalizes $S_i$ and if $T_i=\overline{G}/C_{\overline G}(S_i)$, then $\sigma_i\subseteq\pi(T_i)$.
\end{lemma}
\begin{proof}

Form a chief series of $G$ by refining $1\leq K\leq H\leq G$. 
By Lemma \ref{l:three}, there are at most two factors of this chief series 
that are not disjoint from $\sigma_i$. Since factors between 1 and $K$ and also between $H$ and $G$ are primary groups and $|\sigma_1|=4$, it follows that $|\pi(K)\cap\sigma_i|\cdot |\pi(G/H)\cap\sigma_i|\leq 1$ and there is a unique direct factor $S_i$ of $S$ such that $|\pi(S_i)\cap \sigma_i|\geq 3$. In view of uniqueness, $S_i$ is normal in $\overline G$.

Suppose that $p\in \sigma_i\setminus\pi(S_i)$. If $p\in\pi(G/H)$, then $p\not\in C_{\overline G}(S_i)$, and so $p\in\pi(T_i)$, as required. 

Assume that $p\in\pi(K)$. Let $P$ be a Sylow $p$-subgroup of $K$ and let $H_i$ be the preimage of $S_i$ in $G$.  We derive a contradiction by showing that $pr\in\omega(H_i)$ for some $r\in\sigma_i\cap \pi(S_i)$. By the Frattini argument, $N_{H_i}(P)/N_K(P)\simeq S_i$, so without loss of generality we may assume that $P$ is normal in $H_i$. Also we may assume that $P$ is elementary abelian. By the Schur--Zassenhaus 
theorem, $H_i=P\rtimes M$, where $M\simeq H_i/P$. If $C_{H_i}(P)\not\leq P$,  then $C_{H_i}(P)=P\times Q$, where $Q$ is a normal $\sigma_i'$-subgroup of $H_i$. Factoring it out, we may assume that $M$ acts on $P$ faithfully. 

Let $N=M\cap K$. Suppose that $C_M(N)\leq N$. If $r\in \sigma_i\cap \pi(S_i)$
and $x\in M$ an element of order $r$, then $C_N(x)<N$. So there is $t\in\pi(N)$ and a Sylow $t$-subgroup $T$ of $N$ such that $x$ normalizes but not centralizes $T$. Observe that $t$ is odd, so applying Lemma \ref{l:hh} to the action of $T\rtimes\langle x\rangle$ on $P$, we see that $pr\in\omega(H_i)$.

Thus $C_M(N)N/N=S_i$. Then $C=C_M(N)^{\infty}$ is a perfect central extension of $S_i$. Applying Lemma \ref{l:three_primes}, we conclude that there is $r\in \sigma_i\cap \pi(S_i)$ such that a Sylow $r$\nobreakdash-\hspace{0pt}subgroup of $C$ acts on $P$ with fixed points, and so $pr\in\omega(H_i)$, a contradiction.
\end{proof}

Let $S_i$ be as in Lemma \ref{l:unique}. Observe that $T_i$ is an almost simple group with socle $S_i$. As we noted above, $5\not\in\pi(P_k)$, and so $S_i\not\simeq J_1$.
Also $S_i\not\simeq L_2(u)$  by Lemma \ref{l:lin}. Thus $S_i\simeq {}^2G_2(3^{m_i})$ for some $m_i$. We can take $r_{i2}$ to be a primitive prime divisor of $3^{2p_i}-1$ and one of $r_{i3}$ or $r_{i4}$ to be a primitive prime divisor of $3^{6p_i}-1$, and since $\pi(S_i)$ contains at least two of the numbers $r_{i2}$, $r_{i3}$, $r_{i4}$, it follows that $p_i$ divides $m_i$. But then $m_i=p_i$ by Lemma \ref{l:primes}(d). Thus $S_i\simeq R_i$ for all $i=1,\dots,k$ and we identify them.

It follows that $\overline H\simeq P_k$ since otherwise $2\prod_{i=1}^k r_{i3}\in\omega(G)\setminus\omega(P_k)$. In particular, $K$ is the solvable radical of $G$. So $G/H$ is isomorphic to a subgroup of $\Out(R_1)\times\cdots\times\Out(R_k)$. Now Lemma \ref{l:ree}(b) and Lemma \ref{l:primes}(b) yield  $G=H$. 

Proving that $K=1$,  we may assume, without loss of generality, that $K$ is an elementary abelian $p$-group for some odd prime $p\in\pi(P_k)$. So $\overline G$ acts on $K$. 

Assume first that $p\neq 3,7$. Then there is a unique $i\in\{1,\ldots,k\}$ such that  $p\in\pi(R_i)$. Set $r=r_{i1}$ if $p\not\in\pi(q_i-1)$ and $r=r_{i2}$ otherwise. Applying Lemma~\ref{l:ReeRep}, we see that that $pr\in\omega(G)\setminus\omega(P_k)$, a contradiction.

Let $p=3$ or $7$. Then $p\prod_{i=1}^k r_{i1}\not\in\omega(P_k)$. 
For $i=1,\ldots,k$, take an element $x_i\in R_i$ of order $r_{i1}$. We claim that the element $x=\prod_{i=1}^kx_i$ has a nontrivial fixed point in $K$, so $G$ contains an element of forbidden order. Indeed, by Lemma~\ref{l:ReeRep}, it follows  that $K_1=C_K(x_1)\neq1$. Since $R_2\times\cdots\times R_k$ centralizes $x_1$, the element $x_1'=\prod_{i=2}^kx_i$ acts on $K_1$. By induction on $k$, $C_{K_1}(x_1')\neq1$, and we are done.

Thus, $K=1$, so Theorem~\ref{t:inf} is proved.

\section{Unrecognizable direct powers}

In this section, we will prove Theorem \ref{t:cube}. The proof relies on the following easy lemma. 

\begin{lemma} \label{l:repl}
If $G$ and $H$ are finite groups such that $
\mu(H)\subseteq \mu(G)$  and $|\mu(G)\setminus\mu(H)|<k$,
then $\omega(G^k)=\omega(G^{k-1}\times H)$.
\end{lemma}

\begin{proof}
It suffices to show that  $\mu(G^k)\subseteq \omega(G^{k-1}\times H)$. Every $a\in\mu(G^k)$ can be written as $[m_1,\dots,m_l]$ for $l\leq k$ and some pairwise distinct $m_1,\dots,m_l\in\mu(G)$. If $l<k$, then $a\in\omega(G^{k-1})$. If $l=k$, then at least one of $m_i$ lies in $\mu(H)$, so 
$a\in\omega(G^{k-1}\times H)$. 
\end{proof}

We begin with the proof of Item (b). By Lemma \ref{l:J1}, we have  $\mu(J_1)=\{6,7,10,11,15,19\}$. Set $H=D_6\times D_{10}$, where $D_{2n}$ is the dihedral group of order $2n$. Then $\mu(H)=\{6,10,15\}\subseteq\mu(J_1)$ and $|\mu(J_1)\setminus\mu(H)|=3$, and hence $J_1^4$ is isospectral to $J_1^3\times H$ by Lemma \ref{l:repl}. Since $H$ is solvable, the Mazurov--Shi theorem \cite{12MazShi.t} implies that $J_1^k$ is unrecognizable. 

Proving Item (a), we follow the same lines. 
Recall that $$\mu({}^2G_2(q))=\{6,9, q-1, (q+1)/2, q+\sqrt{3q}+1, q-\sqrt{3q}+1\}.$$ To show that the cube of  ${}^2G_2(q)$ is unrecognizable, it is sufficient to prove the following lemma. 

\begin{lemma} If $q=3^\alpha$, $\alpha\geq 3$ is odd, then there is a group $H$ such that  $O_3(H)\neq 1$ and $\mu(H)=\{6,9,q-1,(q+1)/2\}$.
\end{lemma}

\begin{proof}
As usual, $I_n$ denotes the identity $n\times n$ matrix. We denote by $J_k(\lambda)$ the $k\times k$ Jordan block with eigenvalue $\lambda$. If $g$ is a matrix, then $J(g)$ denotes the Jordan form of $g$.

Let $L=L_2(q)$, let $V$ be the natural 2-dimensional module of $SL_2(q)$ and set $W=V\otimes V^{(3)}$, where $(a_{ij})\in SL_2(q)$ acts on $V^{(3)}$ as $(a_{ij}^3)$ acts on $V$. Then $-I_2\in SL_2(q)$ acts trivially on $W$, and so $W$ is an $L$-module. Denote by $\varphi$ the corresponding homomorphism from $L$ to $GL_4(q)$. It is easy to see that $-I_4\not\in\varphi(L)$, and we set $M=\varphi(L)\times \langle -I_4\rangle$. We claim that $\mu(W\rtimes M)=\{6,9,q-1,(q+1)/2\}$, and so  $W\rtimes M$ is the desired group $H$. 

By Lemma \ref{l:lin}, we have $\mu(L)=\{3,(q-1)/2,(q+1)/2\}$.
Let $g\in L$. If $|g|=3$, then $g$ is conjugate to the image of $J_2(1)$, and so $J(\varphi(g))=J(J_2(1)\otimes J_2(1))=J_3(1)\oplus J_1(1)$. If $|g|\neq 3$, then the characteristic roots of $g$ are $\lambda,\lambda^{-1}$, where $\lambda^{q-1}=1$ or $\lambda^{q+1}=1$.
In this case, the characteristic roots of $\varphi(g)$ are $\lambda^4, \lambda^2, \lambda^{-2}$, and $\lambda^{-4}$.

Observe that $(q-1)/2$ is odd while $(q+1)/2$ is even, and so $\mu(M)=\{6,q-1,(q+1)/2\}$. Let $g\in M$ and $g\neq 1$.
If $|g|=3$, then $g\in \varphi(L)$ and $J(g)=J_3(1)\oplus J_1(1)$ by the preceding paragraph. Hence $\omega(Wg)=\{3,9\}$ by Lemma \ref{l:spectrumVG}. Similarly, if $|g|=6$, then $J(g)=J_3(-1)\oplus J_1(-1)$, and therefore $\omega(Wg)=\{6\}$.

If $|g|$ divides $q-1$ or $(q+1)/2$, then the characteristic roots of $g$ are either $\lambda^4, \lambda^2, \lambda^{-2}, \lambda^{-4}$ or the negatives of these elements. Since $q\equiv 3\pmod 8$, it follows that $8$ divides neither $q-1$ nor $q+1$, so $\lambda^4\neq -1$. If $\lambda^4=1$, then $|g|$ divides 2, so every element of $\omega(Wg)$ divides $6$. If $\lambda^4\neq 1$, then $\omega(Wg)=\{|g|\}$.

Thus $\mu(W\rtimes M)=\mu(M)\cup\{9\}=\{6,9,q-1,(q+1)/2\}$. This completes  the proof of the lemma, and so the proof of Theorem \ref{t:cube} as well.
\end{proof}

We conclude with an observation supplementing Problem \ref{pr:J1}.

\begin{proposition}\label{p:J1}
If $J_1\times J_1\times J_1$ is almost recognizable, then it is recognizable.
\end{proposition}

\begin{proof}
Let $G$ be a group isospectral to $J_1^3$. By the Mazurov--Shi theorem, the hypothesis of the proposition means that $G$ satisfies \eqref{eq:structure}. 

Let $L$ be one of the nonabelian simple groups $L_i$ in \eqref{eq:structure}.
Then  $L$ is among groups  listed in Lemma \ref{l:WalSimple}. Also  $\pi(L)\subseteq\{2,3,5,7,11,19\}$ and $9\not\in\omega(L)$. This implies that $L$ is isomorphic to $J_1$, $L_2(11)$, or $Alt_5$. Note that $\mu(Alt_5)=\{2,3,5\}$ and $\mu(L_2(11))=\{5,6,11\}$.

Thus $\Soc(G)=L_1^k\times L_2^l\times L_3^m$, where $L_1\simeq J_1$, $L_2\simeq L_2(11)$ and $L_3\simeq Alt_5$. Since $|G/\Soc(G)|$ is odd and  $|\Out(L_i)|_{2'}=1$ for all $1\leq i\leq 3$, it follows that  $G=H_1\times H_2\times H_3,$ where $L_1^k\leq H_1\leq L_1\wr S_k$, $L_2^l\leq H_2\leq L_2\wr S_l$, and $L_3^m\leq H_3\leq L_3\wr S_m$.  Set $M=H_2\times H_3$ and let $B_i$ be a complement of $\Soc(H_i)$ in $H_i$. Observe that $|B_i|$ is odd and for every $r\in\pi(B_i)$ and $a\in\omega(L_i)$, we have $ra\in\omega(H_i)$. 

It is clear that $k\leq 3$. If $k=3$ then $M=1$. Also $B_1=1$ since otherwise $|B_1|=3$ and $9\in\omega(G)$. Therefore, in this case $G\simeq J_1^3$. 

Suppose that $k=2$. Then $l\leq 1$, and so $M=L_2^l\times H_3$. Since $19, 7\not\in\pi(M)$, it follows that $\omega(M)$ contains $11$ but not a proper multiple of $11$. This is possible only if $l=1$ and $m=0$. But then $\omega(G)\neq \omega(J_1^3)$.

Suppose that $k=1$. Then $l\leq 2$, and again $M=L_2^l\times H_3$. This time 
$\omega(M)$ contains $11r$ for some $r\in\{7,19\}$ but not a proper multiple of $11r$. It follows that $r$ divides $|B_3|$. Now we see that $22r\in\omega(M)$, regardless of whether $l>1$ or $11\in\pi(B_3)$. 

Finally, let $k=0$. Then $\omega(B_2\times B_3)$ contains $7\cdot 19$ but not $7\cdot 19\cdot 11$, and so $l>1$. It follows that $m=0$ because otherwise  $22\cdot 7\cdot 19\in\omega(G)$. If $7\cdot 19\cdot 5\in\omega(B_2)$, then  $7\cdot 19\cdot 5\cdot 6\in\omega(G)$, which is not the case. Hence an element $x\in B_2$ of order $7\cdot 19$ must centralize an element of order 10 in $\Soc(H_2)$. The centralizer $C_{\Soc(H_2)}(x)$ is isomorphic to some direct power of $L_2(11)$, so if it contains an element of order $10$, then   it contains an element of order $22$ as well. This contradiction completes the proof.
\end{proof}

\textbf{Acknowledgments.} The authors acknowledge the support by the Sobolev Institute of Mathematics state contract (project FWNF-2022-0002).


\begin{thebibliography}{10}

\bibitem{86Bang}
A.~S. Bang, \emph{{Taltheoretiske Unders{\o}gelser}}, Tidsskrift Math.
  \textbf{4} (1886), 70--80, 130--137.

\bibitem{93BrShi}
R.~Brandl and W.~J. Shi, \emph{{A characterization of finite simple groups with
  abelian Sylow $2$-subgroups}}, Ricerche Mat. \textbf{42} (1993), no.~1,
  193--198.
  
\bibitem{85Atlas}
J.~H. Conway, R.~T. Curtis, S.~P. Norton, R.~A. Parker, and R.~A. Wilson,
  \emph{{Atlas of finite groups}}, Clarendon Press, Oxford, 1985.


\bibitem{68Gor}
D.~Gorenstein, \emph{{Finite groups}}, Harper \& Row Publishers, New York,
  1968.

\bibitem{23Survey}
M.~A. Grechkoseeva, V.~D. Mazurov, W.~Shi, A.~V. Vasil'ev, and N.~Yang,
  \emph{Finite groups isospectral to simple groups}, Commun. Math. Stat.
  \textbf{11} (2023), 169--194.
  
\bibitem{67Hup}
B.~Huppert, \emph{{Endliche {G}ruppen. {I}}}, {Grundlehren der Mathematischen
  Wissenschaften}, vol. 134, Springer-Verlag, Berlin, 1967.

\bibitem{82HupBl2}
B.~Huppert and N.~Blackburn, \emph{{Finite groups. II}}, {Grundlehren der
  Mathematischen Wissenschaften}, vol. 242, Springer-Verlag, Berlin-New York,
  1982.

\bibitem{24LMVW}
T.~Li, A.~R. Moghaddamfar, A.~V. Vasil'ev, and Zh. Wang, \emph{On recognition
  of the direct squares of the simple groups with abelian {S}ylow $2$-subgroups},   Ricerche Mat. (2024). \url{https://doi.org/10.1007/s11587-024-00847-8}

\bibitem{12MazShi.t}
V.~D. Mazurov and W.~J. Shi, \emph{A criterion of unrecognizability by spectrum
  for finite groups}, Algebra Logic \textbf{51} (2012), no.~2, 160--162.

\bibitem{61ReeG}
R.~Ree, \emph{{A family of simple groups associated with the simple Lie algebra
  of type $({G}_{2})$}}, Amer. J. Math. \textbf{83} (1961), no.~3, 432--462.

\bibitem{05Vas.t}
A.~V. Vasil'ev, \emph{{On connection between the structure of a finite group
  and the properties of its prime graph}}, Siberian Math. J. \textbf{46}
  (2005), no.~3, 396--404.

\bibitem{69Wal}
J.~H. Walter, \emph{The characterization of finite groups with abelian {S}ylow
  {$2$}-subgroups}, Ann. of Math. (2) \textbf{89} (1969), 405--514.

\bibitem{23YGSV}
N.~Yang, I.~Gorshkov, A.~Staroletov, and A.~V. Vasil'ev, \emph{On recognition
  of direct powers of finite simple linear groups by spectrum}, Ann. Mat. Pura
  Appl. \textbf{202} (2023), no.~6, 2699--2714.

\bibitem{99Zal}
A.~E. Zalesski\u{i}, \emph{Minimal polynomials and eigenvalues of
  {$p$}-elements in representations of quasi-simple groups with a cyclic
  {S}ylow {$p$}-subgroup}, J. London Math. Soc. (2) \textbf{59} (1999), no.~3,
  845--866.

\bibitem{Zs}
K.~Zsigmondy, \emph{{Zur Theorie der Potenzreste}}, Monatsh. Math. Phys.
  \textbf{3} (1892), 265--284.

\end{thebibliography}
\end{document}